\documentclass[12pt]{amsart}
\usepackage{graphicx}
\usepackage{a4wide}
\usepackage{color}
\usepackage{amssymb}
\usepackage{amsthm}
\usepackage{amsmath}
\usepackage{amscd} \usepackage{verbatim}
\usepackage[all]{xy}
\usepackage{mathtools}
\textheight8.5in \textwidth6.5in \numberwithin{equation}{section}
\theoremstyle{plain}
\newtheorem{theorem}{Theorem}[section]

\newtheorem{lemma}[theorem]{Lemma}

\theoremstyle{definition}

\theoremstyle{remark}
\newtheorem*{remarks}{Remarks}
\newtheorem*{remark}{Remark}

\newcommand{\SL}{\text {\rm SL}}
\newcommand{\R}{\mathbb{R}}

\newcommand{\Q}{\mathbb{Q}}

\newcommand{\Z}{\mathbb{Z}}

\newcommand{\leg}[2]{\left( \frac{#1}{#2} \right)}

\newcommand{\im}{\mbox{Im}}

\newcommand{\calL}{\mathcal{L}}

\newcommand{\PSL}{\operatorname{PSL}}

\begin{document}

\title[The partition function modulo 4]{The partition function modulo 4}

\author{Ken Ono}
\address{Department of Mathematics, University of Virginia, Charlottesville, VA. 22904}
\email{ken.ono691@virginia.edu}



\thanks{K.O. thanks  the Thomas Jefferson Fund and the NSF
(DMS-2002265 and DMS-2055118) for their support.}

\begin{abstract}  It is widely believed that the parity of the partition function $p(n)$ is ``random.'' Contrary to this expectation, in this note we prove the existence of infinitely many congruence relations modulo 4 among its values.
For each square-free integer $1<D\equiv 23\pmod{24},$ we construct a weight 2 meromorphic modular form that is congruent modulo 4 to a certain twisted generating function for the numbers
$p\big(\frac{Dm^2+1}{24}\big)\pmod 4$.  We prove the existence of infinitely many linear dependence congruences modulo 4 among suitable sets of holomorphic normalizations of these series.
 These results rely on the theory of class numbers and Hilbert class polynomials, and 
{\it generalized twisted Borcherds products} developed by Bruinier and the author.
\end{abstract}

\maketitle

\section{Introduction and statement of results}

A {\it partition} of a non-negative integer $n$ is a non-increasing sequence of positive
integers that sums to $n.$ The partition function $p(n)$ counts the number of partitions of $n.$
Ramanujan's celebrated congruences  \cite{Ramanujan, RamanujanManuscript} assert, for every non-negative integer $n,$ that
\begin{displaymath}
\begin{split}
p(5n+4)&\equiv 0\pmod 5,\\
p(7n+5)&\equiv 0\pmod 7,\\
p(11n+6)&\equiv 0\pmod{11}.
\end{split}
\end{displaymath}
These congruences have inspired
many works (for example, see \cite{Ahlgren2, AhlgrenOno, Andrews, Atkin1, Atkin2, AtkinObrien,
BringmannOno, 
 HirschhornHunt,  LovejoyOno, Newman1, Newman2, Ono2, Watson} to name a few).
Atkin \cite{Atkin1} and Watson \cite{Watson} notably proved infinite families of congruences, where the moduli are arbitrary powers of 5, 7 and 11.

In the 1960s, Atkin \cite{Atkin2} discovered further congruences, such as
\begin{displaymath}
p(17303n+237)\equiv 0\pmod{13},
\end{displaymath}
where the moduli involve primes other than 5, 7, and 11.
In 2000, the author revisited these examples, and he employed the 
Deligne-Serre theory of modular $\ell$-adic
Galois representations and Shimura's theory of half-integral weight modular forms \cite{Ono2} to prove the existence of infinitely many such congruences for every prime modulus $\ell \geq 5$.
Ahlgren  and the  author \cite{Ahlgren2, AhlgrenOno}  later extended this result to include all moduli coprime to 6.

 In contrast to this rich theory of congruences, much less is known about $p(n)$ modulo 2 and 3.
Apart from a theorem of Radu \cite{Radu}, which proves that there are no congruences of the form
$$
p(an+b)\equiv 0\pmod \ell,
$$
where $\ell \in \{2, 3\},$
very little is known. The parity of $p(n)$ is conjectured to be ``random''.
Indeed, a well-known conjecture of
Parkin and Shanks \cite{ParkinShanks} asserts that
\begin{displaymath}
\lim_{X\rightarrow +\infty} \frac{\# \{ n\leq X \ : \ p(n)\ {\text {\rm is even (resp. odd)}}\}}{X}= \frac{1}{2}.
\end{displaymath}
Although there have been many works (for example, see \cite{Ahlgren1, BoylanOno, GarvanStanton, Hirschhorn1,
Hirschhorn2, HirschhornSubbarao, Kolberg, Mirsky, Nicolas1, Nicolas2, Nicolas3, Nicolas4, Ono1, Ono0,  ParkinShanks, Radu, Subbarao}, to name a
few) on the parity of $p(n)$, this conjecture seems far out of reach.
Indeed, the strongest results do not preclude the possibility that
\begin{displaymath}
\lim_{X\rightarrow +\infty}\frac{\# \{ n\leq X \ : \ p(n) \ {\text {\rm is even (resp. odd)}}\}}{X^{\frac{1}{2}+\varepsilon}}=0.
\end{displaymath}
Even less is known about $p(n)$ modulo 3. Indeed, it isn't known that $p(n)$ is infinitely often a multiple of 3, or takes values in any fixed residue class modulo 3 infinitely often.

In view of these difficult problems, in 2010 the author initiated  \cite{Ono4} research in a different direction. Instead of concentrating on the values of the partition function in arithmetic progressions, he identified special quadratic polynomial arguments.
For each square-free integer $1<D\equiv 23\pmod{24},$ he constructed modular forms that capture the parity of the values $p\big(\frac{D m^2+1}{24}\big).$  Using their modularity, he proved general theorems on the infinitude of the number of even (resp. odd) numbers among such values. We refine this work to the modulus 4.

Throughout, we suppose that $1<D\equiv 23\pmod{24}$ is  square-free.
We define
 the formal twisted generating function\footnote{We note that $p(\alpha)\coloneqq0$ for non-integral $\alpha.$}
\begin{equation}\label{PD}
P(D;q)\coloneqq \sum_{m,n\geq 1} \chi_{-D}(n)\chi_{12}(m)p\left(\frac{Dm^2+1}{24}\right)q^{mn},
\end{equation}
where $\chi_{d}(\cdot)$ is the discriminant $d$ Kronecker character.

\begin{remark}
 Apart from $p(0)=1$, each partition number occurs in a single
$P(D;q),$ as $24n-1=Dm^2$ determines the square-free $1<D\equiv 23\mod 24$. Furthermore, modulo 2 we note that
$$
P(D;q)\coloneqq \sum_{\substack{m,n\geq 1\\ \gcd(n,D)=1}} p\left(\frac{Dm^2+1}{24}\right)q^{mn} \pmod 2.
$$
\end{remark}
 
 \medskip

Our first result establishes that
each $P(D;q)$ is  congruent modulo 4 to a specific weight 2 meromorphic modular form that arises as the logarithmic derivative of a twisted Borcherds product. To make this precise, we require the notion of a Heegner point. Let $\mathcal{Q}_{N,-D}$ be the set of positive definite integral binary quadratic forms $Q(x,y) = ax^2+bxy+cy^2,$ where $N\mid a,$ with discriminant $b^2-4ac = -D.$ 
The congruence subgroup $\Gamma_0(N)$ acts on $\mathcal{Q}_{N,-D}$ with finitely many orbits.
If $Q\in \mathcal{Q}_{N, -D},$ then its {\it Heegner point} is  
\begin{equation}
\tau_Q \coloneqq \frac{-b+\sqrt{-D}}{2a} \in \mathbb{H}\coloneqq\{ \im(\tau)>0\}.
\end{equation}
Since $\Gamma_0(N)$ acts on $\mathcal{Q}_{N, -D}$ with finitely many orbits, there are finitely many such points in any fundamental domain for the action of  $\Gamma_0(N)$ on $\mathbb{H}.$ 

Every $P(D;q)$  is congruent modulo $4$ to the Fourier expansion at $i\infty$ (i.e.  $q\coloneqq e^{2\pi i \tau}$)  of a weight 2 meromorphic modular form with simple poles supported at $\Gamma_0(6)$ discriminant $-D$ Heegner points.

\begin{theorem}\label{Theorem1} If $1<D\equiv 23\pmod{24}$ is square-free, then there is a weight 2 meromorphic modular form $\calL_D(\tau)$ on $\Gamma_0(6)$ for which
$$
P(D;q) \equiv \calL_D(\tau)\pmod 4.
$$
The  poles of $\calL_D(\tau)$  are simple and are  at  $\Gamma_0(6)$ discriminant $-D$ Heegner points.
\end{theorem}

Theorem~\ref{Theorem1} implies the existence of infinitely many linear dependence congruences modulo 4 among $q$-series assembled from the
$P(D;q).$ Although these congruence relations generally require many $q$-series, they are ubiquitous. Therefore, the partition function is not truly random modulo  4, as its values satisfy a complicated infinite web of congruences. 

To make this precise, let $S$ be a finite set of square-free integers $1<D\equiv 23\pmod{24},$ and let $h_S:= \max \left \{ h(-D) \ : \ D\in S\right\},$ the maximum of the class numbers $h(-D).$
For each $D\in S,$ we define the {\it holomorphic normalizations}
\begin{equation}
\widehat{P}_S(D;q)\coloneqq P(D;q)\cdot \Delta(\tau)^{h_S}\cdot H_{-D}(1/\Delta(\tau)).
\end{equation}
For fundamental discriminants $-D<0,$ we recall the {\it Hilbert class polynomial}
\begin{equation}
H_{-D}(X)\coloneqq\prod_{Q\in \mathcal{Q}_{-D}/\PSL_2(\Z)}(X-j(\tau_Q)) \in \Z[X],
\end{equation}
where $\mathcal{Q}_{-D}$ denotes the discriminant $-D$ positive definite integral binary quadratic forms, and
$$
j(\tau)\coloneqq\frac{\left(1+240\sum_{n=1}^{\infty}\sum_{d\mid n}d^3 q^n\right)^3}{q\prod_{n=1}^{\infty} (1-q^n)^{24}}=\frac{E_4(\tau)^3}{\Delta(\tau)}=q^{-1}+744+196884q+\dots
$$
is a hauptmodul for the field of modular functions on $\SL_2(\Z).$ The class number $h(-D)$ is the degree of $H_{-D}(X).$

\begin{theorem}\label{Theorem2}
Assuming the notation and hypotheses above, if
$\# S > 12h_S+2,$
then the $q$-series $\left\{\widehat{P}_S(D;q) \ : \ D\in S\right\}$
 are linearly dependent modulo $4.$
\end{theorem}

\begin{remarks} \ \ \newline
\noindent
(1) Let $S_t=\{D_1, D_2, \dots, D_t\}$ be the set of the first $t$ square-free integers $1\leq D\equiv 23\pmod{24}.$
In earlier work, Griffin, Tsai and the author (see Lemma 2.2 of \cite{GOT}) proved for fundamental discriminants
$-D<0$ that
$$
h(-D)\leq\frac{\sqrt{D}(\log D+2)}{\pi}.
$$
Since a positive proportion of the integers $n\equiv 23\pmod{24}$ are square-free, this implies the existence of a linear dependence relation modulo 4 among the $$\{\widehat{P}_{S_t}(D_1;q), \widehat{P}_{S_t}(D_2;q), \dots, \widehat{P}_{S_t}(D_t;q)\}$$ for every sufficiently large $t$. In fact, there 
is a linear dependence for any subset with more than $12h_{S_t}+2$ many $q$-series.

\smallskip
\noindent
(2) If there are $t>12m+2$ distinct square-free integers $D_i\equiv 23\pmod{24}$ for which $h(-D_i)=m,$ then the proof of Theorem~\ref{Theorem2} shows that the $q$-series
$$\{ P(D_1;q) H_{-D_1}(1/\Delta(\tau)), P(D_2;q) H_{-D_2}(1/\Delta(\tau)),\dots, P(D_t;q) H_{-D_t}(1/\Delta(\tau))\}$$ are linearly dependent modulo 4. This follows from the fact that each $\widehat{P}_{S}(D_i;q)$ has the same factor $\Delta(\tau)^m$ that can be factored out.
\end{remarks}

This note is organized as follows. In Section~\ref{BorcherdsProducts} we recall some twisted Borcherds products that arise from Ramanujan's 3rd order mock theta functions. In Section~\ref{Theorem1Proof} we study the logarithmic derivatives of these products, and  we study their reductions modulo 4. Then we prove Theorem~\ref{Theorem1}. Finally, we prove Theorem~\ref{Theorem2} in Section~\ref{ProofTheorem2}.

\section*{Acknowledgements} \noindent The author thanks Andreas Mono and Badri Pandey for useful conversations related to this paper.

\section{Borcherds products for Ramanujan's 3rd order mock theta functions}\label{BorcherdsProducts}

In their work on derivatives of modular $L$-functions,
the author and Bruinier \cite{BruinierOno} produced {\it generalized twisted Borcherds Products} arising
from weight 1/2 harmonic Maass forms.
These constructions are extensions and generalizations of previous automorphic infinite products obtained by Borcherds
\cite{Borcherds1, Borcherds2} and later Zagier \cite{Zagiertraces}.
The general results (see Theorems~6.1 and 6.2
of \cite{BruinierOno})  give modular forms with twisted Heegner divisor whose
infinite product expansions arise from  weight 1/2 harmonic Maass forms (see  \cite{AMSBook, BruinierFunke} for background on harmonic Maass forms).

Ramanujan's mock theta functions,
which are special examples of holomorphic parts of harmonic Maass forms of weight 1/2 (for example,
see \cite{BringmannOnoMock, BringmannOno, Ono3, Zagier, Zwegers1, Zwegers2}), can be used to construct such products.
We recall one example which involves a vector-valued form assembled from the third order mock theta functions $f(q)$ and $\omega(q),$ which are defined by
\begin{equation}
f(q):=1+\sum_{n=1}^{\infty} \frac{q^{n^2}}{(1+q)^2(1+q^2)^2\cdots (1+q^n)^2},
\end{equation} and
\begin{equation}\label{omega}
\begin{split}
\omega(q):=\sum_{n=0}^{\infty} \frac{q^{2n^2+2n}}{(q;q^2)_{n+1}^2} =
\frac{1}{(1-q)^2}+\frac{q^4}{(1-q)^2(1-q^3)^2}+\frac{q^{12}}{
(1-q)^2(1-q^3)^2(1-q^5)^2}+\cdots.
\end{split}
\end{equation}
It is important to
note that $f(q)$ and $\omega(q)$ have integer coefficients.

Using these mock theta functions, we define the vector-valued function 
\begin{align*}
R(\tau)\coloneqq (R_0(\tau), R_1(\tau),\dots, R_{11}(\tau)),
\end{align*}
with components
\begin{equation}\label{weil}
R_j(\tau)\coloneqq\begin{cases} 0 \ \ \ \ \ &{\text {\rm if}}\ j=0, 3, 6, 9,\\
                    \chi_{-12}(j) q^{-1}f(q^{24}) \ \ \ \ &{\text {\rm if}}\ j=1, 5, 7, 11\\
                                     2q^{8}\left(-\omega(q^{12})+\omega(-q^{12})\right) \ \ \ \ &{\text {\rm if}}\ j=2,\\
                     -2q^{8}\left(\omega(q^{12})+\omega(-q^{12})\right) \ \ \ \ &{\text {\rm if}}\ j=4,\\
                     2q^{8}\left(\omega(q^{12})+\omega(-q^{12})\right) \ \ \ \ &{\text {\rm if}}\ j=8,\\
                     2q^{8}\left(\omega(q^{12})-\omega(-q^{12})\right) \ \ \ \ &{\text {\rm if}}\ j=10.
                     \end{cases}
\end{equation}
For convenience, we denote the coefficients of the components by
\begin{displaymath}
R_j(\tau)\eqqcolon\sum_{n\geq n_j} C_R(j;n)q^n
\end{displaymath}

For each square-free $1<D\equiv 23\pmod{24}$, we have the rational function
\begin{equation}\label{PD}
P_D(X):=\prod_{b\!\!\!\!\mod D} (1-e(-b/D)X)^{\leg{-D}{b}},
\end{equation}
where $e(\alpha):=e^{2\pi i \alpha}$ and $\leg{-D}{b}$ is the Kronecker character for the negative
fundamental discriminant $-D$. We define the {\it generalized Borcherds product}
$\Psi_D(\tau)$ by
\begin{equation}\label{PsiD}
\Psi_{D}(\tau):=\prod_{m=1}^{\infty} P_D(q^m)^{C_R(\overline{m} ; Dm^2)}.
\end{equation}
Here $\overline{n}$ denotes the canonical residue class of $n$ modulo 12.

\begin{theorem}\label{BorcherdsProduct}{\text {\rm [\S8.2 of \cite{BruinierOno}]}}
Assuming the notation and hypotheses above, the function  $\Psi_D(z)$ is a weight 0
meromorphic modular form on $\Gamma_0(6)$ with a discriminant -$D$ twisted Heegner divisor
%%$2Z_{-D,1}\left(\frac{-1}{24},\frac{1}{12}\right)-2Z_{-D,1}\left(-\frac{1}{24},\frac{5}{12}\right)$
(see \S5 of \cite{BruinierOno} for the explicit divisor).
\end{theorem}

\section{Proof of Theorem~\ref{Theorem1}}\label{Theorem1Proof}

We recall the holomorphic differential operator
\begin{align*}
\mathbb{D} \coloneqq q \frac{d}{dq} = \frac{1}{2\pi i}\cdot \frac{\mathrm{d}}{\mathrm{d}\tau}.
\end{align*}
Let $\Psi_{D}(\tau)$ be the generalized Borcherds product in (\ref{PsiD}). 
Then, we define the function
\begin{align*}
\calL_D(\tau) \coloneqq \frac{1}{-\sqrt{-D}} \frac{\left(\mathbb{D} \Psi_{D}\right)(\tau)}{\Psi_{D}(\tau)}.
\end{align*}
We prove the following properties about $\calL_D(\tau)$.
\begin{lemma} \label{lem:logdiff}
Assuming the notation and hypotheses of Theorem \ref{BorcherdsProduct}, the following are true.
\newline
\noindent (1) The function $\calL_D(\tau)$ is a weight 2 meromorphic modular form on $\Gamma_0(6)$ with $q$-expansion
\begin{align*}
\calL_D(\tau) = \sum_{m=1}^{\infty} C_R(\overline{m};Dm^2) m \sum_{\substack{n=1 \\ \gcd(n,D) = 1}}^{\infty} \left(\frac{-D}{n}\right) q^{mn}.
\end{align*}

\noindent (2) The poles of $\calL_D(\tau)$ in the fundamental domain of $\Gamma_0(6)$ are simple and are located at $\Gamma_0(6)$ Heegner points with discriminant $-D.$\end{lemma}

\begin{proof}
Since $\mathbb{D}$ agrees with the Ramanujan-Serre derivative operator in weight $0$, the function $\left(\mathbb{D}\Psi_{D}\right)(g;\tau)$ has weight $2$. 
Together with Theorem \ref{BorcherdsProduct}, this proves the first assertion of part (1). To verify the claimed Fourier expansion, we calculate that
\begin{align*}
\frac{\left(\mathbb{D} P_{D}\right)(q^m)}{P_{D}(q^m)} &= q\frac{d}{dq}\sum_{b \text{ mod } D} \left(\frac{-D}{b}\right) \log\left(1-e^{-2\pi i \frac{b}{D}}q^m\right) = -m\sum_{b \text{ mod } D} \left(\frac{-D}{b}\right) \sum_{n=1}^{\infty} \left(e^{-2\pi i \frac{b}{D}}q^m\right)^n.
\end{align*}
By \cite[(3.12)]{iwakow}, we have
\begin{align*}
\left(\frac{-D}{n}\right) \sum_{b \text{ mod } D} \left(\frac{-D}{b}\right) e^{-\frac{2\pi i b}{D}} = \sum_{b \text{ mod } D} \left(\frac{-D}{b}\right) e^{-\frac{2\pi i bn}{D}} \qquad \text{if } \gcd(n,D) = 1,
\end{align*} 
and both sides of this identity vanish if $\gcd(n,D) > 1$. We use \cite[Theorem 3.3]{iwakow} to evaluate the left hand side of this identity assuming $\gcd(n,D) = 1$, and note that $-D \equiv 1 \pmod{4}$.  Therefore, we find that
\begin{align*}
\calL_D(\tau) &= \frac{1}{-\sqrt{-D}} \sum_{m=1}^{\infty} C_R(\overline{m};Dm^2) \cdot q\frac{d}{dq} \log\left(P_{D}(q^m)\right) \\
&= \sum_{m=1}^{\infty} C_R(\overline{m};Dm^2) m \sum_{\substack{n=1 \\ \gcd(n,D) = 1}}^{\infty} \left(\frac{-D}{n}\right) q^{mn},
\end{align*}
which proves the claimed Fourier expansion of $\calL_D(\tau)$. The second claim follows from Theorem~\ref{BorcherdsProduct} as every zero or pole of $\Psi_{D}(\tau)$ becomes a simple pole of its logarithmic derivative.
\end{proof}

Euler observed that
\begin{displaymath}
P(q)\coloneqq\sum_{n=0}^{\infty}p(n)q^n=\prod_{n=1}^{\infty}\frac{1}{1-q^n}=\frac{q^{1/24}}{\eta(\tau)}.
\end{displaymath}
We have the following alternate identity for $P(q)$, which in turn
provides a crucial congruence between
$P(q)$ and Ramanujan's third order mock theta function $f(q).$

\begin{lemma}\label{pofnhypergeometric}
As formal power series, we have that
$P(q)\equiv f(q)\pmod 4$.
\end{lemma}
\begin{proof}
For every positive integer $m$, the $q$-series
\begin{displaymath}
\frac{1}{(1-q)(1-q^2)\cdots (1-q^m)}=\sum_{n=0}^{\infty}a_m(n)q^n
\end{displaymath}
is the generating function for $a_m(n)$, the number of partitions of
$n$ whose summands do not exceed $m$. Therefore, 
the $q$-series
\begin{displaymath}
\frac{q^{m^2}}{(1-q)^2(1-q^2)^2\cdots (1-q^m)^2}=\sum_{n=0}^{\infty}
b_m(n)q^n
\end{displaymath}
is the generating function for $b_m(n)$, the number of partitions of
$n$ with a Durfee square of size $m^2$. The
identity 
$$
P(q)=1+\sum_{m=1}^{\infty}\frac{q^{m^2}}{(1-q)^2(1-q^2)^2\cdots (1-q^m)^2}
$$
follows by summing in $m$. 
The claimed congruence follows trivially from the definition of $f(q)$ as
$$
\frac{q^{m^2}}{(1-q)^2(1-q^2)^2\cdots (1-q^m)^2}\equiv \frac{q^{m^2}}{(1+q)^2(1+q^2)^2\cdots (1+q^m)^2} \pmod 4.
$$
\end{proof}

\begin{proof}[Proof of Theorem~\ref{Theorem1}]
We  note that $2 (\omega (\pm q^{12})+ \omega(\mp q^{12})\equiv 0\pmod 4$, which then implies that
\begin{displaymath}
R_j(\tau)=
\sum_{n\geq n_j} C_R(j;n)q^n\equiv \begin{cases} \chi_{-12}(j)q^{-1}f(q^{24})\equiv \chi_{-12}(j)q^{-1} P(q^{24}) \pmod 4 \ \ \  &{\text {\rm if}}\ j=1, 5, 7, 11,\\
0\pmod 4 \ \ \ \ \ &{\text {\rm otherwise.}}
\end{cases}
\end{displaymath}
The theorem follows as the definition of $P(D;q)$ is congruent to $\calL_D(\tau)$ modulo 4 thanks to Lemmas~\ref{lem:logdiff} and ~\ref{pofnhypergeometric},  and the fact that
$\chi_{-12}(m)m\equiv \chi_{12}(m) \pmod 4.$ 
 \end{proof}

\section{Proof of Theorem~\ref{Theorem2} }\label{ProofTheorem2}
Here we prove Theorem~\ref{Theorem2}, which follows easily from a well-known theorem of Sturm \cite{Sturm} on congruences between holomorphic modular forms.

\begin{proof}[Proof of Theorem~\ref{Theorem2}]
For each square-free $1\leq D\equiv 23\pmod{24},$ we have that
$\calL_D(\tau)$ is a weight 2 meromorphic modular form on $\Gamma_0(6)$ with simple poles at discriminant $-D$ Heegner points. Therefore, in terms of the modular function $j(\tau)=q^{-1}+744+196884q+\dots$ and the weight 12 cusp form $\Delta(\tau)=q-24q^2+\dots,$ we find that
$$
\calL_D(\tau)\cdot \Delta(\tau)^{h(-D)}\cdot H_{-D}(j(\tau))
$$
is a weight $12h(-D)+2$ holomorphic modular form on $\Gamma_0(6).$ Moreover, since $j(\tau)=E_4(\tau)^3/\Delta(\tau)\equiv 1/\Delta(\tau)\pmod 4,$ we have that 
$$
\calL_D(\tau)\cdot \Delta(\tau)^{h(-D)}\cdot H_{-D}(j(\tau))\equiv \calL_D(\tau)\cdot \Delta(\tau)^{h(-D)}\cdot H_{-D}(1/\Delta(\tau))\pmod 4.
$$

Given the set $S$, we then find that each $\widehat{P}_S(D;q)$ is congruent to a weight $12 h_S+2$ holomorphic modular form on $\Gamma_0(6).$  However, by a theorem of Sturm, two weight $k$  holomorphic modular forms are congruent if their Fourier expansions are congruent for more than the first $$[\SL_2(\Z) : \Gamma_0(6)]\cdot (12 h_S+2)/12= 12h_S+2$$ many terms. 
\end{proof}

\end{document}